\documentclass[11pt]{amsart}

\usepackage{amsmath, amsthm, amssymb}
\usepackage{sum_product}
\usepackage{verbatim} % for \begin{comment} ... \end{comment}
\usepackage[headings]{fullpage} % for 1in margins; cuts off margin comments
\usepackage{setspace}

\author{Esen Aksoy Yazici, Brendan Murphy, Misha Rudnev, and Ilya Shkredov}
\date{December 2015}
\title{Growth Estimates in Positive Characteristic via Collisions}

\address{Esen Aksoy Yazici}
\email{eaksoyzc@gmail.com}

\address{Brendan Murphy, Department of Mathematics, University of
  Rochester, Rochester, NY 14627, USA}
\email{bmurphy8@ur.rochester.edu}

\address{Misha Rudnev, Department of Mathematics, University of Bristol,
Bristol BS8 1TW, United Kingdom}
\email{m.rudnev@bristol.ac.uk}

\address{Ilya Shkredov, Steklov Mathematical Institute, Division of Algebra and Number Theory, ul. Gubkina, 8, Moscow, Russia, 119991 and IITP RAS, Bolshoy Karetny per. 19, Moscow, Russia, 127994}
\email{ilya.shkredov@gmail.com}

\subjclass[2000]{68R05,11B75}

\begin{document}

\begin{abstract} Let $F$  be a field of characteristic $p>2$ and $A\subset F$ have sufficiently small cardinality in terms of $p$. We improve the state of the art of a variety of sum-product type inequalities. In particular, we prove that
$$
|AA|^2|A+A|^3 \gg |A|^6,\qquad |A(A+A)|\gg |A|^{3/2}.
$$
We also prove several two-variable extractor estimates: 
${\displaystyle 
|A(A+1)| \gg|A|^{9/8},}$ $$ |A+A^2|\gg |A|^{11/10},\; |A+A^3|\gg |A|^{29/28}, \; |A+1/A|\gg |A|^{31/30}.$$

Besides, we address  questions of cardinalities $|A+A|$ vs $|f(A)+f(A)|$, for a polynomial $f$, where we establish the inequalities
$$
\max(|A+A|,\, |A^2+A^2|)\gg |A|^{8/7}, \;\; \max(|A-A|,\, |A^3+A^3|)\gg |A|^{17/16}.
$$
Szemer\'edi-Trotter type implications of the arithmetic estimates in question are that a Cartesian product point set $P=A\times B$ in $F^2$, of $n$ elements, with $|B|\leq |A|< p^{2/3}$ makes $O(n^{3/4}m^{2/3} + m + n)$ incidences with any set of $m$ lines. In particular, when $|A|=|B|$, there are  $\ll n^{9/4}$ collinear triples of points in $P$, $\gg n^{3/2}$ distinct lines between pairs of its points, in $\gg n^{3/4}$ distinct  directions.

Besides, $P=A\times A$ determines  $\gg n^{9/16}$ distinct pair-wise distances.

These estimates are obtained on the basis of a new plane geometry interpretation of the  incidence theorem between points and planes in three dimensions, which we call collisions of images.
\end{abstract}

\maketitle

\section{Introduction}
\label{sec:introduction}
In this paper we consider several variations of the sum-product problem over fields $F$ with positive odd characteristic $p$ and multiplicative group $F_*$. All our estimates bear  a constraint in terms of $p$, so the latter should be sufficiently large,  in terms of absolute constants hidden in the forthcoming estimates.

The sum-product problem is roughly to show that any set of numbers is to yield almost maximum possible growth under either addition or multiplication.
Erd\H os and Szemer\'edi \cite{ES} posed this problem and showed that if \(A\) is a finite set of integers, then
\begin{equation}\label{sprod} \max\{|A+A|,|AA|\}\gg |A|^{1+\epsilon},\end{equation}
for some \(\epsilon > 0\), with a small value implicit in their paper.  They conjectured that the above lower bound should hold for any \(\epsilon < 1\).

A significant improvement was made by Elekes \cite{E}, who vindicated  \(\epsilon=1/4\) for \(A\subset\mathbb R\), having brought the Szemer\'edi-Trotter incidence theorem into play.
The current standing ``world record'' over the reals is due to  Konyagin and the fourth listed author \cite{KS}, with \(\epsilon=1/3+\delta\) for some rather small constant \(\delta > 0\), having build upon and added $\delta$ to the result of the construction of Solymosi \cite{So}. The latter construction was also shown to extend to the complex field by Konyagin and the third listed author \cite{KR}.

All the aforementioned results are based on techniques that exploit the order properties of \(\mathbb{R}\)  in a crucial way.

\medskip
The sum-product phenomenon in the positive characteristic case in the context of the prime residue field \(F_p\) was posed and studied in some detail by Bourgain, Katz and Tao \cite{BKT} who proved the existence of some uniform \(\epsilon > 0\) as to inequality \eqref{sprod}, provided that the cardinality $|A|$ of $A\subset F_p$ is sufficiently small in terms of $p$.  Further developing the so-called {\em additive pivot} technique from the latter work, several quantitative estimates on $\epsilon$ have been obtained by various authors, the best one vindicating any $\epsilon<1/11$ in \cite{R0}, see also the references contained therein.

Recently Roche-Newton and the third and fourth listed authors of this paper \cite{RNRS} improved \(\epsilon\) to \(1/5\), using a new geometric method based on a point-plane incidence theorem, due to the third author \cite{R}. This theorem, which we quote as Theorem \ref{mish}, applies over fields of characteristic $p\neq 2$ and carries a constraint on $|A|$ in terms of $p$.

The paper \cite{RNRS} also presented a lower bound \(|A+AA|\gg |A|^{3/2}\), for a sufficiently small $A$. The same bound over the reals follows easily from the Szemer\'edi-Trotter theorem. However, it has only recently been improved  by the fourth listed author \cite{SB}, by some $\delta>0$ adding to the exponent $3/2$.
One of the results here is the matching lower bound of \(|A(A+A)|\gg |A|^{3/2}\). This nearly catches up with the best lower bound over \(\mathbb{R}\), which was pushed slightly beyond \(|A|^{3/2}\) by Roche-Newton and the second and fourth listed authors \cite{MRNS}.

\medskip
This paper proves a variety of new sum-product type estimates in positive characteristic.  They yield a considerable improvement to the state of the art, that arose from a series of questions and estimates by Bourgain, Katz, and Tao \cite{BKT} and further work, e.g., by Bourgain \cite{Bo} and  Bukh and Tsimerman \cite{BT}, see also the references contained therein. These estimates concern, in particular, the minimum cardinality of the sets of values of various polynomials, with several variables in a given set $A$.

Some of these estimates can be interpreted geometrically, as to the point set $P=A\times A$ in the plane  $F^2$, and we obtain rather strong Szemer\'edi-Trotter type theorem for the number of incidences of $P$ with a set of lines in $F^2$.

The thrust of this paper is applying the three-dimensional point-plane incidence Theorem \ref{mish}. The methodological novelty, in comparison to earlier applications of the theorem in, for instance, \cite{RNRS}, \cite{Z}, \cite{BW}, is that we find an efficient way of interpreting the  above incidence in three dimensions  as what we call a {\em collision} in the plane. It turns out that all of the problems discussed in this paper can be cast as event counts of the form $\ell(x)=\ell'(x')$, where $\ell,\ell'$ are lines in some set of lines $L$ in $F^2$ and the abscissae $x,x'$ are in some subset $A$ of  $F$. Such an event gets easily interpreted as an incidence between a point and a plane in space. We thereby obtain estimates, which may be in some sense regarded as positive characteristic analogues of the estimates obtained over the reals via the Szemer\'edi-Trotter theorem. Our estimates are somewhat weaker, for Theorem \ref{mish} is certainly weaker than the Szemer\'edi-Trotter theorem. However, some of them are not too far off the best known results over the reals, for arithmetic applications of the Szemer\'edi-Trotter theorem itself seldom give one optimal results.

We  present a summary of the main results after setting up the notation in the next section.

\subsection{Notation}
\label{sec:notation}
 Throughout this note $F$ is a field of positive odd characteristic $p$, with the multiplicative group $F_*=F\setminus\{0\}$, the sets we are dealing with are finite and non-empty.

Let  $A,B\subset F$ have cardinality $|A|,|B|>1$.  The sum set is defined as
$$
A+B = \{a+b\colon\, a\in A,\,b\in B\}.
$$
Similarly one defines the difference set $A-B$, the product set $AA$, as well as other sets arising from taking algebraic combinations of elements of $A$.
We use the notation $A/B$ for the ratio set, in which case we mean
$$
A/B = \{a/b\colon\, a\in A,\,b\in B\setminus\{0\}\}.
$$
Similarly, to avoid dividing by zero, we later use the notation
 $$(A+B)^{-1} = \{(a+b)^{-1}\colon\,a\in A,\,b\in B,\, a+b\neq 0\}.$$
Besides, for an integer $d$, when we use the power notation $A^d$ for
$$
A^d=\{a^d\colon a\in a\}.
$$

The notation $\E(A,B)$ stands for the so-called additive energy of two sets $A,B \subset F_*$, i.e.,
$$
    \E(A,B) = |\{ a_1+b_1 = a_2+b_2 ~:~ a_1,a_2 \in A;\, b_1,b_2 \in B \}| \,.
$$
If $A=B$ one writes $\E(A)$ instead of $\E(A,A)$,
 calling it the additive energy of $A$.  Estimating the  quantity  $\E(A,B)$ from above is useful, for by the Cauchy-Schwarz inequality on has
 $$
 |A\pm B|\geq \frac{|A|^2|B|^2}{\E(A,B)}.
 $$
 The multiplicative energy $\E^\times$ is defined in the same way, multiplication replacing addition.  We will be using many second moment quantities of energy type, denoted invariably as $\E$ throughout.

We use the standard asymptotic notation. In particular, the symbols $\ll$, $\gg,$ suppress absolute constants in inequalities, as well as respectively do the symbols $O$ and $\Omega$.  $C,\epsilon,\delta$ denote positive constants which may vary from one use to another, the latter two being $<1$ and usually quite small.

\section{Statement of Results}
\label{sec:statement-results} We have attempted to partition our results into three sections. 

These results, by inspection of their proofs, can be easily adapted to several sets $A,B,C,\ldots$ being involved, but for just $A$. In order to demonstrate how different sets come in, we have chosen to formulate the forthcoming Proposition \ref{prop:3} in terms of two sets. Besides, unless stated to the contrary, $A+A$ can be replaced by $A-A$, as well as independently $AA$ by $A/A$. In the following statements we assume $A,B,C \subset F_*$.

\subsection{Sum-product results}
\label{sec:sum-product-results}

Claims in this section extend and improve on earlier results along these lines in \cite{RNRS}.

\begin{proposition}
  \label{prop:2}
For $|A|<p^{3/5}$ one has \[ |A+A|^3|AA|^2\gg |A|^6.\]
\end{proposition}
The same estimate, with the roles of multiplication and addition reversed was proved in \cite{RNRS}.

The next statement is formulated in terms of two sets $A,B$.
\begin{proposition}
\label{prop:3}
Let $A,B\subset F_*$, with $|B|\leq |A|\leq |B|^2$, and $|A|^{1/3}|B|^{4/3}<p$. For  $(\alpha,\beta)\in F^2\setminus\{0,0\}$,  one has \[ |(A+\alpha)/(B+\beta)|^2 |A/B|^3\gg |A|^4|B|^2.\]
\end{proposition}

We also combine results from Proposition \ref{prop:2} with those in \cite{RNRS}.
\begin{corollary}
  \label{cor:5}
One has the estimates
  \begin{enumerate}
\item If $|A|<p^{2/3}$,  then \[ |A+AA|,\, \;|A(A+ A)| \gg |A|^{3/2}.\]
\item If $|A|<p^{5/8}$, then \[ |AA+AA|+|(A+A)(A+A)| \gg |A|^{8/5}.\]
  \end{enumerate}
\end{corollary}

The following Proposition \ref{prop:1} in the special case $d=1$ was proven in \cite{RNRS}.
For  $d>1$, it is a considerable  improvement in a particular case of Theorem 3  of Bukh and Tsimerman \cite{BT}.
\begin{proposition}
  \label{prop:1}
For $|A|<p^{3/5}d^{1/5}$ and an integer $0<d<|A|$, either $|AA|\gg
|A|^2/d$, or
\[
|AA|^3|A^d+A^d|^2\gg |A|^6/d.
\]
\end{proposition}
Note that the second inequality in Proposition \ref{prop:1} cannot hold just under the constraint $d=O(|A|)$.
For example, suppose $A$ is a multiplicative subgroup of $F_*$ of rank $d$. Then the second inequality alone is false.

\subsection{Line geometry estimates}
In this section we summarise new results in geometry of lines defined by the plane point set $P=A\times A$.
Let $\T(A)$ denote the number of \emph{collinear triples} of points of $P$. Let also
$$ R(A) := \left\{\frac{(a-b)(c-d)}{(a-c)(b-d)}:\;a,b,c,d\in A\right\}\;\subset FP^1$$
be the set of cross-ratios defined by $A$, viewed as the subset of the projective line $FP^1$, and $R_\infty(A)$ the same set pinned at $c=\infty$.

The next proposition is the  main result in this section.

\begin{proposition}
\label{prop:7}
 If $|A|\ll p^{2/3}$ then $\T(A)\ll |A|^{9/2}$.
\end{proposition}

Its implications are as follows.
\begin{corollary}\label{cor:n} If $|A|\ll p^{2/3}$, the set $P=A\times A$  be a $n$-point set in $F^2$. Then \begin{enumerate}
\item $P$ determines $\Omega(n^{3/2})$ distinct lines between pairs of its points, in $\Omega(n^{3/4})$ distinct directions.  For every $q\in P$, but possibly one point, there are $\Omega(n^{6/10})$ distinct lines supporting $q$ and some other point of $P$.
\item $P$ forms $O(n^{3/4}m^{2/3}+m)$ incidences with a set of $m$ lines in $F^2$.
\item The  set $A\in FP^1$ determines $\Omega(|A|^{3/2})$ distinct cross-ratios, pinned at any point of $A$.\end{enumerate}\end{corollary}

As a link to the results in the forthcoming Section \ref{sec:prod-transl} we also have the following.
\begin{corollary}\label{cor:y} If $2\leq |A|\ll p^{2/3}$, then for at least $|A|/2$ values of $a\in A$, 
$$|A(A-a)|\gg |A|^{5/4}.$$\end{corollary}

For other applications of the incidence bound it may be useful to record an easy generalisation of the incidence bound of Corollary \ref{cor:n} to the case of $P=A\times B$.
\begin{corollary}\label{cor:m} Let $P=A\times B$ of $n$ points, with $|B|\leq |A|<p^{2/3}$. Then $P$ has $O(n^{9/4} + |A|^3|B|)$ collinear triples of points and makes $O(n^{3/4}m^{2/3}+m+n)$ incidences with any $m$ lines.\end{corollary}

For the most economical way of deriving analogues of these results over $\R$ and $\C$, see the paper of Solymosi and Tardos \cite{ST}, where the corresponding estimates combine basic enumerative combinatorics with the fact that affine transformations preserve order on the real line.   Our  estimate on the number $\T(A)$ in positive characteristic  is weaker  by roughly the factor of $|A|^{1/2}$.  In a sense, it has come a long way, being based on the point-plane incidence theorem in three dimensions over the algebraic closure of $F$ in \cite{R}.

Earlier Szemer\'edi-Trotter type applications of the additive pivot technique developed in \cite{BKT} and \cite{Bo} are much weaker. In \cite{BKT} it was established that there exists some $\delta>0$, such that the number of incidences between $n$ lines and $n$ points in $F_p^2$, with a sufficiently small $n$ vs $p$, was $O(n^{3/2-\delta})$. We refer to their claim as the {\em qualitative} version of the Szemer\'edi-Trotter theorem in positive characteristic. 

Further build-up on \cite{BKT} and \cite{Bo} led, in particular,  to the estimate $\Omega(|P|^{269/267})$ (vs the exponent $3/2$ here) for the number of  the number of distinct lines determined by $P$ in a paper of Helfgott and the third listed author \cite{HR}, which was improved by Jones to $110/109$ in \cite{J}.

\begin{remark} Cross-ratios appear to play an important role in sum-product type questions. A cross-ratio count was recently used by Iosevich, Roche-Newton and the third author \cite{IRR} to improve the exponent $2/3$ for the number of distinct nonzero wedge products, defined by a point set in $\R^2$ to $9/13$. An improvement on the lower bound $|R(A)|=\Omega(|A|^{3/2})$ would enable an improvement of the exponent $2/3$ (proven in \cite{R})  in positive characteristic as well. The true lower bound for $|R(A)|$ is most likely $|A|^3$, possibly modulo logarithmic terms in $|A|$. It is remarkable that over $\R$ one has a near-optimal bound $\Omega(|A|^2/\log|A|)$ for $|R_\infty(A)|$ but nothing better than $\Omega(|A|^2)$ for the size of the full set $R(A)$. Both results are implicit in the paper of Solymosi and Tardos \cite{ST} and  explicit in Jones' paper \cite{J1}. It is also remarkable that the smallest value of $|R(A)|$  is presumably achieved when $A$ is a geometric progression, while the number of collinear triples in $A\times A$, the second moment quantity  with the support $|R_\infty(A)|$ is maximum when $A$ is an arithmetic progression. \end{remark}

\subsection{Products of translates $A(A+\alpha)$}
\label{sec:prod-transl}

The results in this section  deal with a {\em two-variable extractor} -- a function of only two variables in $A$, whose range is considerably larger than $|A|$. The difference with the claim of Corollary \ref{cor:x} is that a slightly weaker estimates holds for {\em any} nontrivial translate of $A$. It is an improvement on the current state of the art due to Zhelezov \cite{Z} on the two-variable extractor introduced by Bourgain \cite{Bo} and incrementally improved in the recent past by several authors, see the references in \cite{Z}.
\begin{proposition}
\label{prop:4}
For $|A|<p^{8/13}$ and $\alpha\in F_*$, one has \[| A(A+\alpha)| \gg |A|^{9/8}.\]
\end{proposition}

As a corollary, we have the following bound on the number of representations of $\alpha$ as a difference in $A-A$, where $A$ has a small product set. 
The second statement follows from Corollary \ref{cor:n} (b).
\begin{corollary}
  \label{cor:3}
If $|A|<p^{8/13}$
then for any $\alpha \neq 0$ one has
$$
    |A \cap (A+\alpha )| \ll |AA|^{8/9}\,.
$$
Moreover, for $|A|< p^{2/3}$ the following holds $|A \cap (A+\alpha )| \ll |AA|^{4/3} |A|^{-1/2}$.
\end{corollary}
This can be considered as a generalisation of Stepanov's method to ``almost-subgroups'', see \cite{Stepanov}.

\subsection{Sums of polynomial images}
\label{sec:sums-polyn-imag}
The question of relating the size of the sumset $|A+A|$ vs $|f(A)+f(A)|$, for some function $f$ has been studied in many variants. Over the reals, the prominent case is that of  a convex $f$ after Elekes, Nahtanson and Ruzsa \cite{ENR} observed that the Szemer\'edi-Trotter theorem can be used to study it. See also \cite{SS}, \cite{LRN} for the current state of the art over $\R$.

In positive characteristic general but somewhat weak estimates for a polynomial $f$ were established by Bukh and Tsimerman \cite{BT} over $F_p$, still largely based on the additive pivot approach.

The following Proposition improves a special case of these estimates, and also includes two two-variable extractors.
\begin{proposition}
\label{prop:5}
 For $|A|<p^{3/5}$, one has
 \begin{enumerate}
  \item \[ |A+A|^3 |A^2+A^2|^4 \gg |A|^8,\qquad |A+A^2|\gg |A|^{11/10};\]
  \item \[|A-A|^3 |A^3+A^3| \gg |A|^{17/4},\qquad |A+A^3| \gg |A|^{29/28}.\]
 \end{enumerate}
\end{proposition}

The claim (b) of Proposition~\ref{prop:5} is the only one in this paper that is worse if $A-A$ is replaced by $A+A$ (for instance, by using the Ruzsa distance inequality); difference sets are preferred because the proof involves taking the discrete derivative.

Proposition~\ref{prop:5} yields explicit quantitative bounds for the positive characteristic version of the Erd\H os distance problem in $F^2$, for both quadratic (or  ``Euclidean'') and cubic ``distances''.
\begin{corollary}
  \label{cor:4}
Fix $A\subset F$ and let $P=A\times A\subset F^2$.
\begin{enumerate}\item If $|A|<p^{8/15}$ then
\[
|(A-A)^2 +(A-A)^2| \gg |A|^{9/8},
\]
that is, the point set $P$ defines $\Omega(|P|^{9/16})$ distinct quadratic distances.
\item If $|A|<p^{7/12}$ then
\[
|(A-A)^3 +(A-A)^3| \gg |A|^{36/35},
\]
that is, the point set $P$ defines $\Omega(|P|^{36/70})$ distinct cubic distances.
\end{enumerate}
\end{corollary}
The Erd\H os distinct distance problem, that is that the set of $n$ points in the plane determines $\Omega(n^{1-o(1)})$ distinct distances, was resolved over the reals by Guth and Katz \cite{GK}. The only previously available estimate for the number of distinct distances in small plane sets over prime residue fields that we are aware of is contained in \cite{BKT}, Section 7, with the exponent $1/2+\epsilon$, for some $\epsilon>0$, whose existence follows from the qualitative positive characteristic  version of the Szemer\'edi-Trotter theorem. We present a quantitative estimate with $\epsilon = 1/16$.

As far as cubic or higher degree ``distances'' in $\R^2$ are concerned, the estimate $\Omega(|P|^{2/3})$  follows from the Szemer\'edi-Trotter theorem for curves.

\subsection{Sums of reciprocals}
\label{sec:sums-reciprocals}

Our last set of estimates was motivated by Theorem 4 in \cite{Bo}, which implies the inequality
\[
|A+A|+|1/A+ 1/A| \gg|A|^{1+\epsilon},
\]
for some $\epsilon>0$ and, say $|A|<p^{1/2}$.
The latter estimate also relied on the above-mentioned qualitative positive characteristic  version of the Szemer\'edi-Trotter theorem from \cite{BKT}. Once we have it quantitatively as stated in the claim (b) of  Corollary \ref{cor:n}, we can vindicate $\epsilon = 1/15$.
\begin{proposition} \label{prop:8}
For $|A|<p^{5/8}$ one has
$$
|A+A|+|1/A+ 1/A| \gg|A|^{16/15}, \qquad |A+A^{-1}| \gg |A|^{31/30}.
$$
\end{proposition}
As an analogue of the first statement in Corollary \ref{cor:3} we have the following bound for a number of points of $A\times A$ on a hyperbola $xy=\alpha$, for $A$ with small additive doubling. This may be viewed as an ``almost-interval''  generalisation of the much stronger results of Cilleruelo and Garaev \cite{GC} for intervals.
\begin{corollary} \label{cor:x} If $|A|<p^{5/8}$
then for any $\alpha \neq 0$ one has
$$
    |A \cap \alpha/A | \ll  |A+A|^{15/16} \,.
$$\end{corollary}

We can  also produce a slightly stronger three, rather than two variable extractor involving the reciprocals.
\begin{proposition}
  \label{prop:6}
For $|A|<p^{8/13}$, one has \[ |1/A + 1/(A+A)| \gg |A|^{9/8}.\]
\end{proposition}

\subsection{Additional Remarks}
Before we continue with proofs, let us say a few words about the results.

\subsubsection{The relative strength of our results} Our results are a considerable improvement of what was known before Theorem \ref{mish} became available, some only slightly weaker than what has been established over the reals by the Szemer\'edi-Trotter theorem.  However, Theorem \ref{mish} is  weaker than the Szemer\'edi-Trotter theorem. This accounts for the fact that all the two-dimensional implications of the three-dimensional Theorem \ref{mish} are drawn in the specific setting of the  point set being a Cartesian product. This fact also expresses itself in a technical tool, the forthcoming Theorem \ref{main}: rather than providing an upper bound the number of incidences between a set of points and a set of lines in the plane, it gives one on the number of what we call collisions, that is events $\ell(a)=\ell'(a')$, where $\ell,\ell'$ are in a set of lines and $a,a'$ in a set of abscissae.

For instance, a consequence of the Szemer\'edi-Trotter theorem is Beck's theorem, which for a Cartesian product  $P=A\times A\subset \mathbb R^2$, states that the number of distinct lines connecting pairs  of point of $P$ is $\Omega(|A|^4)$; this, in turn, implies that the lines go in $\Omega(|A|^2)$ distinct directions.
We establish a similar result over a field $F$ of positive characteristic: Corollary~\ref{cor:n} states that a point set $P=A\times A\subset F^2$ determines $\Omega(|A|^3)$ distinct lines in $\Omega(|A|^{3/2})$ different directions.
However, Corollary~\ref{cor:n} \emph{only}\/ applies in the case where $P$ is a Cartesian product.
For previously established results over $F_p$ see \cite{HR} and \cite{J}.

\begin{comment}
There are two special cases to sum-product type estimates  that have been studied in the context of $F_p$ rather extensively. The first one is asking how large is $A+A$ when $A$ is a subgroup of the multiplicative group of $F_p$. The best partial results here are due to Vyugin and the fourth author \cite{VS}. The other extreme case is asking how large is $AA$ when $A$ is an arithmetic progression. In the latter case Cilleruelo, Garaev, Bourgain, et al.  succeeded in obtaining essentially best possible estimates in a variety of cases, including the size of $AA$,  as well as the sum of reciprocals $A^{-1}+A^{-1}$. See, e.g. \cite{GC}, \cite{BG} and references therein. The methods developed for the above two special cases are rather specific. It appears to be a reasonable question whether one can benefit by attempting to combine them with the technique in this paper.  
\end{comment}

\subsubsection{The regime of small sets} Our results concern the case of ``small sets''. Smallness is defined in terms of the characteristic $p$ of $F$.  If $F=F_p$, one may ``glue'' our  estimates with what has or may be established for ``large sets'' using, e.g. exponential sums over a finite field $F_q$, where $q=p^d$. Of course, one can only ``glue'' when $q=p$, for our results are constrained in terms of $p$. Proving a more specific version of Theorem \ref{mish} over a finite extension $F_q$ of $F_p$, where the constraint would involve $q$ rather than $p$ appears to be a difficult problem, see the relevant remarks in \cite{R}, \cite{RNRS}.

The large set estimates over $F_q$ tend to trivialise for $|A|<q^{1/2}$.
For instance, if $|A|=O(q^{2/3})$ then one has the bound $\Omega(|A|^3/q)$ for the number of distinct distances defined by the point set $A\times A$ (see, e.g. \cite{BHIPR}, Theorem 1.6), which becomes, with $q=p$, stronger than Corollary \ref{cor:4} for $|A|>p^{8/15}$. Coincidentally or not, this is precisely the constraint for Corollary \ref{cor:4}.

The same type of argument in \cite{RNRS} regarding sum-product estimates (\cite{RNRS}, Remark 7) shows that in the context of $F=F_p$ the estimates in Propositions \ref{prop:2} and \ref{prop:5}, with $d=1$ become weaker than what is known in the ``large set'' regime for $|A|>p^{5/8}$. The latter constraint does not appear in the statement of the propositions explicitly but arises throughout the proofs, in particular, as the constraint for the sum-product inequality \eqref{speq}.

Similarly for Corollary \ref{cor:n}, an elementary argument shows any set  $P$ of $|P|>q$ points in the finite plane $F_q^2$ determines a line in every direction. See e.g. \cite{Al} or  \cite{IRZ}, Section 2.  Moreover, it is also known that for $|P|\geq (1+\epsilon)(q+1)$, $|\mathcal L(P)|\gg q^2$, see \cite{bunch}. Thus in the context of $F=F_p$, the statements of Corollary \ref{cor:n} about the line set $\mathcal L(P)$ is only interesting with $|A|<p^{1/2}$.

\section{Lemmata}
\label{sec:lemmata}
All our results are applications of the following incidence theorem from \cite{R}.
\begin{theorem}  \label{mish} Let $Q, \Pi$ be sets of points and planes, of cardinalities respectively $m$ and $n$, in the projective space $FP^3$, with $m\leq n$ and $m=O(p^2)$. Let $k$ be the maximum number of collinear points.
Then
\begin{equation}\label{pups}
    |I(Q,\Pi):=\{(q,\pi)\in Q\times\Pi:\,q\in \pi\}|=O( n\sqrt{m} + kn) \,.
\end{equation}
\end{theorem}
In the forthcoming applications we will always have $m=n$, otherwise the estimate of Theorem \ref{mish} gets worse.\footnote{Without additional assumptions the estimate \eqref{pups} cannot be improved in positive characteristic: take, e.g., $F=F_p$, the prime residue field, $Q$ the unit sphere and $\Pi$ the set of all planes in $F_p^3$. See also \cite{R}, Remark 4.}

Let us also quote a minor variation of Theorem \ref{mish}, where points and planes have weights, see Theorem 15 in \cite{R}. This will be used in the proof of  Proposition \ref{prop:1}.

The set-up is a pair $(Q,\Pi)$ comprising a finite set of points and a finite set of planes in $FP^3$. One has a positive real-valued weight function $w$ on $Q\cup \Pi$, with the supremum-norm $w>0$ and $L_1$-norm $W$. Define
\begin{equation}\label{wgt}
I_w = \sum_{q\in Q,\pi\in \Pi} w(q)w(\pi) \delta_{q\pi},\qquad \mbox{with }\;\delta_{q\pi}=\left\{\begin{array}{ll} 1,&q\in\pi\\0,&q\not\in\pi.\end{array}\right.
\end{equation}
That is every incidence $(q,\pi)$ counts with weight $w(q)w(\pi)$.

\addtocounter{theorem}{-1}

\renewcommand{\thetheorem}{\arabic{theorem}*}

\begin{theorem}  \label{mishw} Let $k$ be the maximum number of collinear points, counted without weights, and $W/w_0<p^2$. One has
$I_w=O( W^{3/2} w^{1/2} + k w W)$.
\end{theorem}

\renewcommand{\thetheorem}{\arabic{theorem}}

We will combine the above geometric incidence estimates, with standard tools of additive combinatorics.
Namely, throughout the paper  we use the well-known Pl\"{u}nnecke-Ruzsa inequalities, see \cite{Ruzsa_book}, \cite{Ruzsa_card} or \cite{TV}.

For additive sets $A,B,C$ one has the Ruzsa distance inequality
\begin{equation}\label{rd}
|A-B|\leq \frac{|A+C||B+C|}{|C|}.
\end{equation}

The next lemma summarises the statements of the Pl\"unnecke-Ruzsa inequalities, in the three variants that we use.

\begin{lemma}\label{l:plunnecke-ruzsa}
Let also $A,B,C$ be finite set of an abelian group such that $|A+B|\le K|A|.$
Then for an arbitrary $0<\delta<1$ there is a nonempty set $X\subseteq A$ such that
$|X| \ge (1-\delta) |A|$ and for any integer $k$ one has
\begin{equation}\label{f:plunnecke-ruzsa1'}
    |X+kB|\le (K/\delta )^{k}  |X| \,.
\end{equation}

In addition,
\begin{equation} \label{2pl}
|A+B| \leq \frac{|A+C||B+C|}{|C|},\end{equation}

and for all positive integers $n,m$,
\begin{equation} \label{fatpl}
|nB-mB| \leq K^{n+m}|A|. \end{equation}
\end{lemma}

\section{Collisions and the image set theorem}
In this section we prove a technical tool that we refer to as the image set theorem. Even though it is not used to establish every instant of our main results, the reasoning within its proof presents a clear roadmap to the structure of the argument in the whole paper.

All the results in this paper follow from a general claim about the energy of a set of lines in the form $y=ax+b$ and a set of values of $x$. We refer to the event $y=y'$ as a  collision.
The goal of this section is to state and prove this result.
\subsection{Collisions}
\label{sec:ideas-behind-main}
To prove a lower bound for \(|A+AA|\), the authors of \cite{RNRS} prove an upper bound for the number of solutions to the equation
\begin{equation}
  \label{eq:1}
  a+bc=a'+b'c',\qquad a,\ldots,c' \in A
\end{equation}
by applying Theorem \ref{mish} to the set of planes defined by linear equations \(x+by=a'+zc'\).
The same method does not seem to work at the first glance for \(A(A+A)\), since fixing any three variables in the equation
\begin{equation}
  \label{eq:2}
  a(b+c)=a'(b'+c') \qquad a,\ldots,c' \in A
\end{equation}
leads to a quadratic equation or a degenerate set of lines.

If one looks more carefully, it does, for the sets \(A+AA\) and \(A(A+A)\) are both \emph{image sets of lines}.
To realise this for \(A+AA\), consider the set \(L\) of lines of the
form \(\ell(x)=ax+b\) with \(a\) and \(b\) in \(A\).
Then \[L(A)=\{\ell(a)\colon \ell\in L, a\in A\}=A+AA.\]
For \(A(A+A)\), the appropriate lines have the form \(\ell(x)=a(x+b)\).

In this language, equations \eqref{eq:1}, \eqref{eq:2} both count the number of \emph{collisions of images}
\begin{equation}
  \label{eq:3}
  \ell(x)=\ell'(x')
\end{equation}
with \(\ell\) and \(\ell'\) in \(L\) and \(x\) and \(x'\) in \(A\).
We will use \(\E(L,A)\) to denote the number of solutions to \eqref{eq:3}, since this quantity is analogous to additive or multiplicative energy.

Thus the upper bound for \eqref{eq:1} generalises to \emph{any set of lines}, in particular the set of lines that yield \(A(A+A)\) as the image set.
The general upper bound has the form
\begin{equation}
  \label{eq:4}
  \E(L,A)\ll (|L||A|)^{3/2} + k|L||A|,
\end{equation}
where \(k\) is the maximum of \(|A|\) and the largest number of lines of \(L\) contained in a pencil (that is, a family of concurrent or parallel lines).

To prove \eqref{eq:4}, we simply fix \(\ell\) and \(x'\) in equation \eqref{eq:3}. This yields a linear equation in \(x\) and \(\ell'\), where we view \(\ell'\) as a point by projective duality. Then we apply the point-plane incidence bound \eqref{pups} as before.

This upper bound yields a general lower bound for the size of the image set of a family of lines
\begin{equation}
  \label{eq:5}
  |L(A)|\gg \min\{ \sqrt{|L||A|}, |L||A|k^{-1}\},
\end{equation}
which yields the lower bounds of \(|A+AA|, |A(A+A)|\gg |A|^{3/2}\).

Thus the same method works for $A+AA$ and  \(A(A+A)\) because covector sets with elements $(a,b)$ defining lines  \(\ell(x)=ax+b\) and $(a,ab)$ defining lines \(\ell(x)=a(x+b)\) both form a $|A|\times |A|$ grid.

\medskip
Before giving the general statement of the energy bound, we introduce formally the notation for set of lines, image sets of lines, and the energy $\E(L,A)$ associated to a set of lines $L$ and a set of numbers $A$.

Let $P\subset F^2\setminus\{(0,0)\}$ .
 We define
 $$L(A)=L_P(A)=\{\ell_{m,b}(a)=ma+b\colon   (m,b)\in P, \,a\in A\}$$
where $L$ is a set of lines in the form $y=mx+b$ in $F^2$, that is
 $$L=L_P=\{\ell_{m,b}: (m,b)\in P  \}.$$
As long as for $(m,b)\in P$ one has $m\neq 0$, one can regard $L_P$ as a subset of the affine group $F_*^2$ over $F$.
 Clearly  $|L_P|=|P|.$

Hence, generalising \eqref{eq:1}, \eqref{eq:2} to the case of three sets $A,B,C$ we have
 \begin{eqnarray}\label{products}
 BA+C&=&\{ba+c\colon a\in A,\, (b,c)\in B\times C\}\\
 &=&L_P(A), \nonumber
 \end{eqnarray}
 where $P=B\times C$.
  \begin{eqnarray}\label{sums}
 B(A+C)&=&\{ba+bc\colon a \in A, (b,c)\in B\times  C\}\\
 &=&L_P(A), \nonumber
  \end{eqnarray}
  where $P=\{(b, bc)\colon b\in B, c \in C\}$. Note that the set of pairs of points $(b,bc)$ is a $|B|\times|C|$ grid. Namely, the former points lie at the intersections of lines through the origin with slopes $c$ and vertical lines $x=b$.  By a $m\times n$ grid we mean the intersection of two line plane pencils of $m$ and $n$ points, respectively.

  In both examples above we would like to replace the notation $L_P$ with  $L_{B\times C}.$

In the sequel we write $L_P\simeq L_{B\times C}$ for a set of lines when the set $P$ of covectors defining the lines in $L_P$ is {\em projectively equivalent} to a $|B|\times |C|$ grid. This,  in particular implies that $P$ has at most $\max(|B|,|C|)$ collinear points -- something we will have to watch apropos of the parameter $k$ in Theorem \ref{mish}.

In this context we  use more energy notation
$$\E(L,A)=|\{(\ell, \ell', a, a')\in L^2\times A^2:\ell(a)=\ell'(a')\}|.$$
Here if  $L=L_P$, then
$$\E(L,A)=|\{(ma+b=m'a'+b': (m,b), (m, b')\in P;\, a,a'\in A\}|.$$

\subsection{Image set  theorem}
\label{sec:upper-bound-ELA}
We are now ready to  state and prove the energy bound which we call the image set theorem.
\begin{theorem}\label{main}
Let $P\subset F_*^2$, and  $A\subset F_*$, such that $|P||A|=O(p^2)$. Let $L=L_P\simeq L_{B\times C}$. Then
\begin{equation}
\label{enest}
\E(L,A)\ll |L|^{\frac{3}{2}}|A|^{\frac{3}{2}}+k|L||A|,
\end{equation}
where
\begin{equation}
  \label{eq:11}
  k \leq\max\{|A|,|B|,|C|\}.
\end{equation}

Hence
\begin{equation}
\label{sizest}
|L(A)|\gg \min\{{\sqrt{|L||A|},|L||A|k^{-1}}\}.
\end{equation}
\end{theorem}
We refer to the former term in the estimates of Theorem \ref{main} as the main term; in all the forthcoming applications it will dominate the second term.

We record an immediate corollary which  subsumes Theorem 3 and Corollary 4 in \cite{RNRS}, which did provide the following bound for the set
$BA+C$.
\begin{corollary}
  \label{cor:1}
Given $A,B,C\subseteq F$, let $M=\max(|A|,|B|,|C|)$. Then
\[
|BA+C|,\;|B(A+C)|\gg\min\left(\sqrt{|A||B||C|}, M^{-1}|A||B||C|, p \right).
\]
\end{corollary}

\begin{proof}
Let $P=\{(b,bc)\colon b\in B, c\in C\}$, so that
\[
B(A+C)=L_P(A),
\]
as in \eqref{sums}.
Since $L_P\simeq L_{B\times C}$, Theorem~\ref{main} implies that
\begin{equation}
  \label{eq:14}
  |B(A+C)|=|L_P(A)|\gg\min\{{\sqrt{|B||C||A|},|B||C||A| k^{-1}}\},
\end{equation}
where $k=M\leq\max\{|A|,|B|,|C|\}$ and $|A||B||C|<p^2$.

If $|A||B||C|>p^2$, then we may refine to subsets $A'\subseteq A, B'\subseteq B,$ and $C'\subseteq C$ such that $|A'||B'||C'|\approx p^2$.
Applying \eqref{eq:14} to $A',B',$ and $C'$ yields
\[
|B(A+C)|\geq |B'(A'+C')|\gg p.
\]

This completes the proof of the lower bound for $|B(A+C)|$.
The proof of the lower bound for $|BA+C|$ is similar.
\end{proof}

\subsection{Proof of Theorem \ref{main}}
\label{sec:proof-theor-main}

Now we give the proof of Theorem \ref{main}, following the sketch in Section \ref{sec:ideas-behind-main}.
\begin{proof}
Let
\begin{equation}
  \label{eq:6}
  Q=\{(m,b,a'): (m,b)\in P,\,a'\in A\},
\end{equation}
and
\begin{eqnarray}
\label{eq:7}
 \Pi&=&\{\pi: ax+y=m'z+b': (m',b')\in P,\, a\in A\}\\
 &=&\{\pi: ax+y-m'z=b': (m',b')\in P,\, a\in A\}\nonumber
 \end{eqnarray}
be sets of points and planes in $F^3$, respectively.

Then
$$|Q|=|\Pi|=|P||A|=|L||A|$$ and since $L_P\simeq L_{B\times C}$, the maximum number of collinear points or planes is $$k\leq\max\{|B|,|C|,|A|\}$$ where $B$ and $C$ are defined as before.

It follows from Theorem \ref{mish} that
 $$|I(S,\Pi)|=\E(L,A)\ll (|L||A|)^{3/2}+ k|L||A|,$$
which proves the first part of the theorem.

Now let $$r_{L(A)}(y)=|\{((m,b), a)\in P\times A: y=ma+b\}|.$$ Then by the Cauchy-Schwarz inequality
\begin{eqnarray*}
|L|^2|A|^2&=&\left(\sum_{y}r_{L(A)}(y)\right)^2\\
&\leq&|L(A)|\sum_{y}(r_{L(A)}(y))^2\\
&=&|L(A)| \E(L,A)\\
&\ll&|L(A)| [ (|L||A|)^{3/2}+k|L||A|].
\end{eqnarray*}
Therefore,
$$|L(A)|\gg \min\{\sqrt{|L||A|}, |L||A|k^{-1}\},$$
and the proof is complete.
 \end{proof}

  % \end{proof}

\subsection{Technical remarks}
\label{sec:extens-energy-bound}
We make a few technical notes as to the  forthcoming proofs.

\subsubsection{Weighted version of (\ref{enest})}

\begin{comment}
\begin{theorem}
  \label{thm:weighted-enest}
Suppose the point set $Q$ defined in \eqref{eq:6} and the set of planes $\Pi$ defined in \eqref{eq:7} have total weight $W$ and maximum weight $w$. Then assuming that $W\ll wp^2$, we have 
\begin{equation}
  \label{eq:8}
  \E(L,A)\ll W^{3/2}w^{1/2}+kwW.
\end{equation}
\end{theorem}
In particular, if $L$ is a set of lines with total weight $W_L$ and maximum weight $w_L$, and $A$ is a subset of $F$ with total weight $W_A$ and maximum weight $w_A$, then assuming that $W_LW_A\ll w_Lw_Ap^2$, we have
\begin{equation}
  \label{eq:9}
  \E(L,A)\ll (W_LW_A)^{3/2}(w_Lw_A)^{1/2}+kW_LW_Aw_Lw_A.
\end{equation}
This is simply \eqref{eq:8} with $W=W_LW_A$ and $w=w_Lw_A$.
\end{comment}

\renewcommand{\thetheorem}{\arabic{theorem}}
Theorem \ref{mishw} enables one to use the energy bound \eqref{enest} of Theorem \ref{main} with weighted sets of points and lines.
We will do this on a case-by-case basis, rather than formulate a general theorem.
See the proofs of Propositions \ref{prop:1}, \ref{prop:5}, and \ref{prop:6} for examples.

\subsubsection{Bounding the constant $k$ from Theorem \ref{mish} in collision estimates}

The upper bound \eqref{eq:11} for the maximum number $k$ of collinear points or planes can be formulated in a number of ways.

The most general condition is that $k$ is less than the max of $|A|$ and number of lines of $L$ contained in a pencil. Recall that a pencil of lines is a set of lines whose defining covectors are collinear. Thus,
\[
k\leq \max\{|A|, M\}
\]
where $M$ is the maximum number of lines in $L$ whose co-vectors are collinear.
Said differently, $M$ is the maximum number of concurrent or parallel lines in $L$.

If the set of lines $L$ is projectively equivalent to a grid, $L\cong L_{B\times C}$, then
\[
M\leq\max\{|B|,|C|\},
\]
since projective grids have at most $|B|$ or $|C|$ lines in a pencil.

The following upper bound does not require $L$ to be projectively equivalent to a grid; it says that $M$ is less than either (1) the number of slopes of lines in $L$, or (2) the largest number lines with a given slope.
To state this precisely, suppose that $L=L_P$ for some subset $P$ of $F_*\times F$, and let $S$ be the set of slopes of lines of $L$ (that is, $S$ is the set of first coordinates of $P$). Then
\begin{equation}
  \label{eq:12}
  M\leq\max\left\{|S|, \max_{m\in S}|\{(m,b)\in P\}|\right\}.
\end{equation}
We will use this estimate in the proof of Proposition \ref{prop:5}.

\section{Proofs of the Main Results}
\label{sec:proofs-sum-product}

All the proofs  follow a common pattern: to bound the ``energy'' of a two-variable expander, we multiply by 1 or add 0 to introduce more variables. We then interpret the problem as a collision count and bound the resulting expression using Corollary \ref{cor:1} of Theorem \ref{main} as a tool, or  Theorems \ref{mish}, \ref{mishw} if the scenario is slightly more unwieldy (rather than generalising Theorem \ref{main} to account for it). When we end up with lower bounds for multiple sumsets we use the Pl\"unnecke-Ruzsa inequalities to go down in the number of summands.

\subsection{Sum-product applications}
\label{sec:sum-prod-pf}

These results are stated in Section \ref{sec:sum-product-results}.

\subsubsection{Proof of Proposition \ref{prop:2}}

Recall that Proposition \ref{prop:2} states that for $|A|<p^{3/5}$, one has \[ |A+A|^3|AA|^2\gg |A|^6.\]

\begin{proof}
We will prove that
\begin{equation}
  \label{eq:10}
  \E^\times(A)\ll |A||A+A|^{3/2}
\end{equation}
under suitable hypotheses.
Once we establish \eqref{eq:10}, Proposition \ref{prop:2} follows by Cauchy-Schwarz:
\[
\frac{|A|^4}{|AA|}\leq \E^\times(A)\ll  |A||A+A|^{3/2}\implies |A+A|^3|AA|^2\gg |A|^6.
\]

To prove \eqref{eq:10}, first we bound the multiplicative energy by an average over an expression with more variables.
$$\begin{aligned} \E^\times(A) & = |\{ ab = a'b' ~:~ a,\dots, b'\in A \}| \\ & \le |A|^{-2} | \{ (a+c)b -cb  = (a'+c')b'- c'b'  ~:~ a,\ldots, c' \in A \}|. \\
& \leq
 |A|^{-2}| \{ sb -cb  =s'b'- c'b'  ~:~ b,b',c,c'\in A;\, s,s' \in A+A \}|. \end{aligned}$$
Next, we note that expression
\[
\{ sb -cb  =s'b'- c'b'  ~:~ b,b',c,c'\in A;\, s,s' \in A+A \}|
\]
equals $\E(L,A+A)$, where $L=L_P$ is the set of $|A|^2$ lines defined by \[P=\{(b,cb)\colon b,c\in A\}.\]
As in \eqref{sums}, we have $L_P\simeq L_{A\times A}$, thus Theorem~\ref{main} implies that
 \begin{equation}\label{app}
 \E(L,A+A)\ll  |A|^3|A+A|^{3/2} + |A+A|^2|A|^2,
 \end{equation}
assuming that $|A|^2|A+A|<p^2$.

The second term of \eqref{app} can be discarded, since otherwise $|A+A|\gg |A|^2$.
To deal with the constraint $|A|^2|A+A|<p^2$, note that Proposition \ref{prop:2} is satisfied trivially if $|A+A|>|A|^{4/3}$, and otherwise if $|A+A|\leq |A|^{4/3}$, then the constraint is satisfied if $|A|<p^{3/5}$, since $|A|^2|A+A|\leq |A|^{10/3}$.

All together, we see that if $|A|<p^{3/5}$, then either Proposition \ref{prop:2} is trivially true, or \eqref{eq:10} holds, thus the proof is complete.
 \end{proof}
The proof implies that for the slightly higher threshold of $|A|<p^{5/8}$, one has the inequality
\begin{equation}
 |A+A|+|AA|\gg|A|^{6/5},
\label{speq}\end{equation}
since then Theorem \ref{main} can be applied as above under the stronger constraint $|A+A|<|A|^{6/5}$.
 Inequality \eqref{speq} was already established in \cite{RNRS} with the roles of multiplication and addition reversed in its proof.

\begin{remark}
    It was proved in \cite{RNRS} that
\begin{equation}\label{f:14.12.2015_0}
    \E (A,B) \ll (|A| |BC|)^{3/2} |C|^{-1/2} + M |A| |BC| |C|^{-1} \,,
\end{equation}
    where $M=\max\{ |A|, |BC|\}$ and $C$ is any set such that $|A||C||BC| < p^2$.
The proof of Proposition \ref{prop:2} gives us a dual inequality for the multiplicative energy
\begin{equation}\label{f:14.12.2015_1}
    \E^\times (A,B) \ll (|A| |B+C|)^{3/2} |C|^{-1/2} + M |A| |B+C| |C|^{-1} \,,
\end{equation}
    where $M=\max\{ |A|, |B+C|\}$ and $C$ is any set with  $|A||C||B+C| < p^2$.

Finally, there is another new bound for the energy (see the proof of Proposition \ref{prop:6} below)
\begin{equation}\label{f:14.12.2015_2}
    \E^\times (A,B) \ll (|A| |A(B+\alpha)C|)^{3/2} |C|^{-1/2} + |A| |A(B+\alpha)C|^2 |C|^{-1} \,,
\end{equation}
provided that $|A||C||A(B+\alpha)C| < p^2$, and  $\alpha \neq 0$ is any number.
We think that the estimates (\ref{f:14.12.2015_1}), (\ref{f:14.12.2015_2}) will be useful in further sum--products estimates (applications of bound (\ref{f:14.12.2015_0}) can be found in \cite{RNRS} and \cite{Z}).

In addition, Proposition \ref{prop:6} -- see the forthcoming proof -- provides a strong bound for an important characteristic of a set, that is the number $\mathrm{T} (A)$ of {\it collinear triples} in $A\times A$. See e.g. \cite{TV}, \cite{KS}, \cite{Sh_diff} for various applications of the number $\T(A)$. The quotient is obviously connected with sum--products questions in view of a simple formula $\T(A) = \sum_{a,b\in A} \E^\times (A-a,A-b)$.
\end{remark}

\subsubsection{Proof of Proposition \ref{prop:3}}

Recall that Proposition \ref{prop:3} states that for $A,B\subset F_*$, such that  $|B|\leq |A|\leq|B|^2$, $|A|^{1/3}|B|^{4/3}<p$ and $(\alpha,\beta)\in F^2\setminus\{0,0\}$,  one has \begin{equation}\label{need} |(A+\alpha)/(B+\beta)|^2 |A/B|^3 \gg |A|^4|B|^2.\end{equation}

\begin{proof}
The proof is similar to  Proposition \ref{prop:2}.  Let $\E$ be the number of solutions of the equation
  $$
  \frac{a+\alpha}{b+\beta} =   \frac{a'+\alpha}{b'+\beta},\qquad a,a'\in A;\, b,b'\in B\setminus\{-\beta\}.
  $$
  Rearranging yields
  $$
  ab'+ a\beta + b'\alpha =  a'b+ a'\beta + b\alpha.
  $$
Assuming without loss of generality that $\alpha\neq 0$, we replace $a=c'd'/d'$ and $a'=cd/d$, with $d,d'\in B$. Then $\E$ is bounded by $|B|^{-2}$ times the number of solutions of the following equation:
  $$
  e(b+\beta)d + b\alpha =   e'(b'+\beta)d' + b'\alpha~:~ b,d,b',d'\in B;\,e,e'\in A/B.
  $$
  The latter quantity equals $\E(L, A/B)$, where the set of lines $L=L_P$, with
  $$P=\{(\,(b+\beta)d, b\alpha):\,b,d\in B\}$$ is projectively equivalent to a  $|B|\times  |B|$ grid.

Applying Corollary  \ref{cor:1}, we obtain that given that $|B|^2|A/B|<p^2$,
$$
\E\ll |B|^{-2} \left( |B|^3 |A/B|^{3/2} + |A/B|^2|B|^2 \right) \ll |B| |A/B|^{3/2},
$$
provided that $|A/B|<|B|^2$. Otherwise the claim of the proposition holds trivially, for the left-hand side of \eqref{need} can be bounded from below by $|A|^2|B|^6\geq |A|^4|B|^2.$

Then by Cauchy-Schwarz,
$$
|(A+\alpha)(B+\beta)|^2 |A/B|^3 \gg |A|^4|B|^2.
$$
It remains to observe that if $|A/B|\gg |A|^{2/3}|B|^{2/3}$, then inequality \eqref{need} holds trivially. Otherwise the constraint $|B|^2|A/B|<p^2$ becomes $|A|^{1/3}|B|^{4/3}<p$, as claimed in the statement of the proposition.

Also observe that once we assume that $|A/B|\ll |A|^{4/5} |B|^{2/5}$, the constraint $|B|^2|A/B|<p^2$ improves to $|A|^{2/5}|B|^{6/5} < p,$ and inequality \eqref{need} still holds. We thus conclude that
\begin{equation}
 |A/B|+|(A+\alpha)/(B+\beta)|\gg |A|^{4/5}|B|^{2/5},\qquad \mbox{if }\;|A|^{2/5}|B|^{6/5} < p,
\label{speq1}\end{equation}
as an analogue of  \eqref{speq}.  \end{proof}

\subsubsection{Proof of Corollary \ref{cor:5}}
Recall that Corollary~\ref{cor:5} states that
  \begin{enumerate}
\item if $|A|<p^{2/3}$,  then \[ |A+AA|,\, \;|A(A+ A)| \gg |A|^{3/2},\]
\item and if $|A|<p^{5/8}$, then \[ |AA+AA|+|(A+A)(A+A)| \gg |A|^{8/5}.\]
  \end{enumerate}
The proof of Corollary  \ref{cor:5} is an application of Corollary \ref{cor:1}.

\begin{proof}
Part (a) of Corollary~\ref{cor:5} follows immediately from Corollary~\ref{cor:1} by setting $A=B=C$.

To prove part (b), we apply Corollary \ref{cor:1} to the sets $A+A, A, A$ and the sets $A, A, AA$ to obtain
\begin{equation}
  \label{eq:15}
  |(A+A)(A+A)|\gg\min\{|A|\sqrt{|A+A|}, |A|^2, p\}
\end{equation}
and
\begin{equation}
  \label{eq:16}
    |AA+AA|\gg\min\{|A|\sqrt{|AA|}, |A|^2, p\},
\end{equation}
assuming that
\begin{equation}
|A|^2\max(|A+A|,|AA|)< p^2.
\label{cnst}\end{equation}

Now, either $|AA|+|A+A|\gg |A|^2$, and there is nothing to prove, or one has
\begin{equation}\label{alm}
|(A+A)(A+A)|  \gg |A|\sqrt{|A+A|},\qquad |AA+AA|  \gg |A|\sqrt{|AA|}.
\end{equation}
By the sum-product estimate  \eqref{speq}, for $|A|<p^{5/8}$ one has
$$
\max(|A+A|,|AA|)\gg |A|^{6/5}.
$$
The claim of Corollary \ref{cor:5} now follows, after invoking the sum-product estimate  \eqref{speq}, provided that \eqref{cnst} is satisfied. If it is not satisfied, we may refine $A$, throwing away element by element, until \eqref{cnst} becomes satisfied as an equality up to a constant. Then inequalities \eqref{alm} become true for the refinement $A'$ of $A$ in question, with the right-hand side being replaced by $p\gg|A|^{8/5}.$

\end{proof}

\subsubsection{Proof of Proposition \ref{prop:1}}

Recall that Proposition \ref{prop:1} states that for $|A|<p^{3/5}d^{1/5}$ and an integer $0<d<|A|$, either $|AA|\gg
|A|^2/d$, or
\[
|AA|^3|A^d+A^d|^2\gg |A|^6/d.
\]

\begin{proof}
Let
$$
\E_d(A) =  |\{ a^d+b^d = c^d+e^d~:~ a,b,c,e\in A \}|.
$$
We have, with $f,g\in A$:
 \begin{equation}\begin{aligned}
  \E_d(A) & =  |A|^{-2} |\{ a^d+(bf)^d/f^d  =  c^d+(eg)^d/g^d \}| \\
  &\leq |A|^{-2} | \{ a^d + h^d / f^d  = c^d + k^d/g^d ~:~ a,f,c,g\in A;\, h,k \in AA \}|.\end{aligned}
  \label{rear}\end{equation}
 The latter equation counts incidences between multi-sets of points in $F^3$, with coordinates $(a^d, 1/f^{d}, k^d)$ and planes with equations $x+ h^d y - (1/g^d)z = c^d$.  The points and planes can have weights (multiplicities), due to the fact that one can have up to $d$ solutions to the equation say $a^d=t$ in $A$. The maximum number $k$ of collinear points or planes is $|AA|$, the total weight $W=|A|^2 |AA|$ and the maximum weight is $w=d$. Applying Theorem \ref{mishw}, assuming $|A|^2|AA|\ll dp^2$, yields
 $$
   \E_d(A) \ll |A|^{-2} ( \sqrt{d} |A|^3 |AA|^{3/2}+ d|A|^2|AA|^2).
 $$
Assuming that $|AA| < |A|^2/d$ we can disregard the second term in the latter estimate.

By Cauchy-Schwarz one has
 $$
 |A^d+A^d| \gg |A|^4/ \E_d(A),
 $$
 whence Proposition \ref{prop:1} follows, provided that
 $$
 |A|^2|AA|<p^2d.
 $$
Note that if $|AA|>|A|^{4/3}d^{1/3}$ then Proposition \ref{prop:1} is true trivially, for $|A^d+A^d| \geq |A|/d$. Otherwise, to comply with the above constraint we need $|A|^2 |A|^{4/3}d^{1/3} < p^2d$, as stated in the proposition.
\end{proof}

\subsection{Line geometry estimates}

Let $\T(A)$ denote the number of collinear triples in the point set  $P=A\times A$ in $F^2$.
The first nontrivial bound for $\T(A)$ in finite fields setting was proved in \cite{Sh_diff}.
Here we have a much stronger result.

\subsubsection{Proof of Proposition \ref{prop:6}}
Recall that Proposition \ref{prop:6} claims that if $|A|<p^{2/3}$, one has $\T(A)\ll |A|^{9/2}.$

\begin{proof} Let $P=A\times A$. The collinearity of three points $(a,a'),(b,b'),(c,c')\in P$, where we assume $a,c\neq b$; $a',c'\neq b'$, is expressed by the condition
\begin{equation}
\det\left( \begin{array}{ccc} 1 & 1& 1\\ a & b & c\\ a' & b' & c' \end{array}\right) = 0,
\label{zdet}\end{equation}
which means $(a-b)/(c-b)=(a' -b')/(c' -b')$. Thus
\begin{equation}
  \label{eq:13}
  \T(A) \ll \left|\{(a,\ldots,c')\in A^6\colon \frac{a-b}{c-b}=\frac{a'-b'}{c'-b'}\neq 0,\infty\}\right|+|A|^4.
\end{equation}

Let $L=L_P$ where
\[
P=\{(1/(c-b), -b/(c-b))\colon c,b\in A\}.
\]
That is, $L$ is the set of lines of the form \[\ell_{c,b}(t)=\frac{t-b}{c-b},\]
with $c$ and $b$ in $A$.

By \eqref{eq:13} and our definition of $L$, it follows that $\T(A)\ll \E(L,A)+|A|^4$.
Proposition~\ref{prop:7} will follow from Theorem \ref{main} if we can show that $|L|=|A|^2$ and $k\leq |A|$, since then
\[
\E(L,A)\ll (|A|^3)^{3/2}+|A|^4 \ll |A|^{9/2}.
\]

First $|L|=|A|^2$, since every $(x,y)\in P$ corresponds to a unique pair $(c,b)$ in $A\times A$, where
\[
b=-\frac yx \quad\mbox{and}\quad c =\frac 1x -\frac yx.
\]
Second, to show that $k\leq |A|$ we must show that at most $k$ points of $P$ are collinear.
Consider the linear equation $Ax+By=C$ with $A,B,$ and $C$ fixed; suppose one of $A,B$  equals $1$.
Plugging in $x=1/(c-b)$ and $y=-b/(c-b)$ yields the equation
\[
A-Bb=C(c-b),
\]
which has at most $|A|$ solutions $(b,c)$, as required.
\end{proof}

\subsubsection{Proof of Corollary \ref{cor:n}}

Let $\mathcal L(P)$ as the set of lines connecting pairs of distinct points of $P=A\times A$ and $R(A)$ the set of cross-ratios defined by $A$.
There are two statements we wish to prove, assuming $|A|<p^{2/3}$:
\begin{enumerate}
\item $P$ determines $\Omega(n^{3/2})$ distinct lines between pairs of its points, in $\Omega(n^{3/4})$ distinct directions.  For every $q\in P$, but possibly one point, there are $\Omega(n^{6/10})$ distinct lines supporting $q$ and some other point of $P$.
\item $P$ forms $O(n^{3/4}m^{2/3}+m)$ incidences with a set of $m$ lines in $F^2$.
\item The  set $A\in FP^1$ determines $\Omega(|A|^{3/2})$ distinct cross-ratios, pinned at any point of $A$.
\end{enumerate}

\begin{proof}
In Statement (a), the claim about distinct directions follows from Corollary  \ref{cor:1}, replacing $B(A+C)$ by  for a set $(A-A)(A-\alpha)^{-1}$, for any $\alpha$.
The statement  about the minimum size of the set $\mathcal L(P)$ follows from Proposition \ref{prop:6} and the H\"older inequality:
$$
|A|^4 \ll \sum_{l\in \mathcal L(P)} n^2(l) \ll |\mathcal L(P)|^{1/3} \T(A)^{2/3},
$$
where $n(l)$ is the number of points of $P$ supported on line $l\in \mathcal L(p)$.

Moreover, let $(\alpha, \beta)\in A\times A$ be such that $m=|(A-\alpha)/(A-\beta)|$ is minimal. If $m>|A|^{6/5}$, there is nothing to prove. Otherwise, let $A'= A-\alpha $, $B'=B-\beta$, and we can apply formula \eqref{speq1} in the form:
$$
\forall (\alpha', \beta')\neq (0,0), \qquad \left|(A'-\alpha')/(B'-\beta')\right|\gg |A|^{6/5}.
$$
This completes the proof of climes in statement (a).

\medskip The statement (b) on the number of incidences between $P$ and a set of $m$ lines in $F^2$ follows from Proposition \ref{prop:6} in the standard way. The Proposition implies that for $k\geq 3$, the numbers $x(k)$ of lines supporting $k$ or more points is $O(n^{9/4}/k^3)$. Given $m$ lines, we arrange them by non-decreasing popularity $k$ in terms of the number of points $k$ of $P$ supported on a line. The $x$th line on the list supports, inverting the above estimate, $k(x)\ll n^{3/4}x^{-1/3}$ points. Integrating, the number of incidences between lines indexed by $x=1$ and $x=x_0$ is $O(n^{3/4}x_0^{2/3})$, and $x_0\leq m$. In this count we have not included lines with two or fewer points of $P$, and their contribution is at most $2m$.

Finally, statement (c) also follows from Proposition \ref{prop:6} by observing that the quantity  $r=(x-a)/(z-a)$ appearing in \eqref{eq:13} is a cross-ratio of $x,a,\infty$, and $z$ and applying the Cauchy-Schwarz inequality:
$$
|R(A)|\geq |R_\infty(A)| \geq  |A|^6/\T(A),
$$
where $R_\infty(A) $ is the set of cross-ratios pinned at infinity.

Note that $\infty$ does not have to be in $A$, for any $a\in A$ can be mapped to $\infty$ by a M\"obius transformation $\frac{1}{z-a}$, which leaves the set $R(A)$ invariant.
\end{proof}

\subsubsection{Proof of Corollary \ref{cor:y}} Recall that Corollary \ref{cor:y} states that If $2\leq |A|\ll p^{2/3}$, then for at least $|A|/2$ values of $a\in A\subset F_*$, $|A(A-a)|\gg|A|^{5/4}.$ 

\begin{proof} Let $A'\subset A$ be such that for every $a\in A'$, 
\begin{equation}\E^\times(A,A-a)  \geq \left|  \left\{ \frac{b}{a-c} = \frac{b'}{a-c'}:\;b,b',c,c' \in A\setminus\{a\}\right\}\right|\;\geq \; C|A|^{11/4},\label{let}\end{equation} for some sufficiently large $C$. Suppose, $|A'|> |A|/2$. If this is brought to a contradiction, the claim of the corollary follows by Cauchy-Schwarz inequality for every $a\in A\setminus A'.$

The assumption \eqref{let}  implies that the set of at most $m=2C^{-1}|A|^{9/4}$ lines  bushing at points $(a,0)\in F^2$ with $a\in A'$, is such that the  union of these lines supports at least  $C|A|^{15/4}/16$ pairs of points of $P=A\times A$.

However, as in the proof of (b) of the previous corollary, if the lines are arranged as a list by non-decreasing popularity $k(x)$ in terms of the number of points $k$ of $P$ supported on the $x$th line, the $x$th line on the list supports
only $O(|A|^3x^{-2/3})$ pairs of points of $P$. Integrating to $x_0=2C^{-1} |A|^{9/4}$ we obtain a contradiction for a suitably large $C$. \end{proof}

\subsubsection{Proof of Corollary \ref{cor:m}} \begin{proof}This is a straightforward generalisation of the corresponding arguments in the proof of Proposition \ref{prop:6} and Corollary \ref{cor:n}. 

In equation \eqref{zdet} one has now $a,b,c\in A$, $a',b',c' \in B$. For $B$ sufficiently small versus $A$ there is an extra term to be included in the bound:  \eqref{zdet} has $|A|^3|B|$ trivial  solutions, corresponding to the case $a'=b'=c'$. Thus the number of solutions of \eqref{zdet} is bounded, by Cauchy-Schwarz, as
$$O(\sqrt{\T(A)\T(B)} +  |A|^3|B|) = O( n^{9/4} + |A|^3|B|),$$
by Proposition \ref{prop:6}.

The contribution of the additional term plays in the incidence estimate is easy to take into account as $O(n)$, for it corresponds to horizontal lines only, whose number is at most $|B|$.  \end{proof}

\subsection{Products of translates \(A(A+\alpha)\)}

\subsubsection{Proof of Proposition \ref{prop:4}}
Recall that Proposition \ref{prop:4} states that for $|A|<p^{8/13}$ and $\alpha\in F_*$, one has \[| A(A+\alpha)| \gg |A|^{9/8}.\]

 \begin{proof}
Suppose, $|A|>C$ for some large enough absolute $C$.
Set $S=A(A+\alpha)$, $|S|=K|A|$.  Use Lemma \ref{l:plunnecke-ruzsa} to refine $A$ to a large subset $A'$ (say with cardinality $.99|A|$), so that
 $|A'(A'+\alpha)(A'+\alpha)| \ll K^2|A|.$  Without loss of generality  $-\alpha\not\in A'$. Replace $A$ with $A'$ using the same notations $A,S$.

 Let $\E$ be the number of solutions of the equation
$$
a(b+\alpha)= a'(b'+\alpha)~:~ a,b,a',b'\in A.
$$
Write, for $c\in A$:
$$
\alpha^2 + a(b+\alpha) = (a+\alpha)  (b+\alpha) - \alpha b = c(a+\alpha)  (b+\alpha)/c - \alpha b.
$$
It follows that for $Z= A(A+\alpha)(A+\alpha)$, we have $\E$ bounded by $|A|^{-2}$ times the number of solutions of the equation
$$
z/c -\alpha b = z'/c' -\alpha b'   ~:~ b,c,b',c'\in A;\,z,z' \in Z.
$$
That is we have $\E\leq|A|^{-2}\E(L,Z)$, where $L\simeq L_{A\times A^{-1}}$.  Assume that
$K^2|A|^3<p^2$, which is necessary to apply Theorem \ref{main}/Corollary \ref{cor:1}. It is easy to see that  the main term in its estimate dominates because otherwise $K \gg |A|^{1/2}$ and the result follows immediately.

Thus
$$
\E\ll |A|^{-2} |A|^{9/2} K^{3}.
$$
Then by Cauchy-Schwarz
\begin{equation}\label{fin}
K^8 \gg |A|,
\end{equation}
and Proposition \ref{prop:4} follows,  for assuming that $|A|<p^{8/13}$ and that \eqref{fin} fails implies that the applicability threshold $K^2|A|^3<p^2$ of Theorem \ref{main}/Corollary \ref{cor:1} holds, which ensures \eqref{fin}.
\end{proof}

\subsubsection{Proof of Corollary \ref{cor:3}}

Recall that Corollary \ref{cor:3} states that if $|A|<p^{8/13}$ and $|AA| \le M|A|$ for some real number $M\ge 1$, then for any $\alpha \neq 0$ one has
$$
    |A \cap (A+\alpha )| \ll  (M|A|)^{8/9} \,,
$$
and if $|A|<p^{2/3}$ then $|A \cap (A+\alpha )| \ll M^{4/3} |A|^{5/6}$.

\begin{proof}
  To prove 
  the first part of 
  Corollary \ref{cor:3}  we set
$A' = A \cap (A+\alpha)$.
Then $A'\subseteq A$ and $A'-\alpha \subseteq A$.
Thus $A'(A'-\alpha) \subseteq AA$.

Using the condition of the Corollary and Proposition \ref{prop:4}, we obtain
$$
    |A'|^{9/8} \ll |A'(A'-\alpha)| \le |AA| \le M|A| \,,
$$
whence $|A'| \ll (M|A|)^{8/9}$.

To prove the second part let us just note that the number of solutions of the equation $a_1- a_2 = \alpha$, $a_1,a_2 \in A$ equals $|A\cap (A+\alpha)|$ and can bounded as
$$
    |A\cap (A+\alpha)| \ll
        |A|^{-2} |\{ p a^{-1} - p_* a^{-1}_* = \alpha : p,p_* \in AA,\, a,a_* \in A \}| \,.
$$
Applying Corollary \ref{cor:n} (b) with $A=A^{-1}$, $n=|A|^2$, $m=M^2 |A|^2$, we see that the last quantity is bounded by $M^{4/3} |A|^{5/6}$ as required.
\end{proof}

\subsection{Sums of images of polynomials} In the forthcoming proofs we use Theorem \ref{mish} rather than Corollary \ref{cor:1}, for the parameter sets defining the family of lines $L$ are not projectively equivalent to Cartesian products. Moreover, strictly speaking it should be the weighted Theorem \ref{mishw}, due to branching of the squares and cubes, but we bypass this, as the number of branches  is two or three, so the weights are $O(1)$. Modulo this, the general strategy remains the same.
\label{sec:sums-imag-polyn}

\subsubsection{Proof of Proposition \ref{prop:5}}

Recall that claim (a) of Proposition \ref{prop:5} states that for $|A|<p^{3/5}$, one has \begin{equation}\label{quote} |A+A|^3 |A^2+A^2|^4 \gg |A|^8,\qquad |A+A^2|\gg |A|^{6/5}.\end{equation}

  \begin{proof}[Proof of Claim (a)] Suppose, $|A|>C$ for some large enough absolute $C$.
 Set $S=A^2 + A^2$, $|S|=K|A|$. Use Lemma \ref{l:plunnecke-ruzsa} to refine $A$ to a large subset $A'$, so that
\begin{equation}\label{useplu}
|{A'}^2+{A'}^2 + {A'}^2| \ll K^2|A|.\end{equation}
Replace $A$ with $A'$ using the same notations $A,S$.

 Let $\E$ be the number of solutions of the equation
\begin{equation}\label{3eq}
c^2+ a^2-b^2 =  {c'}^2+ {a'}^2-{b'}^2,
\end{equation}
  with variables in $A$. Rewrite the left-hand side as
  $$
 c^2 + 2as - s^2,\qquad s=a+b.
  $$
 Thus $\E = \E(L,A)$ where a single line in the set of lines $L=L_P$, where
 $$P=\{(2s, c^2-s^2):\, s\in A+A,\, c\in A.\}$$
 With our notation we cannot write
$L\simeq L_{(A+A)\times A^2}$, for the sets $P$ is  not projectively equivalent to a $|A+A|\times |A^2|$ grid. However, as it is within the proof of Theorem \ref{main}, $\E$ is the number of incidences between planes and points in $F^3$, where the set of, say points is $P\times A$.

Assuming $|A|^2|A+A|< p^2$ we apply Theorem \ref{mish}, which gives the same estimate as \eqref{app} above.
In the application of Theorem \ref{mish} within Theorem \ref{main} a point/plane can, strictly speaking, have an integer weight, for both $\pm c$ yield the same element of $A^2$. So, weights are bounded by $2$, and this gets absorbed into constants. The main term in the right-hand side of the estimate \eqref{app} reappearing here,  once again dominates (or there is nothing to prove), that is
$$\E\ll |A|^{3}|A+A|^{3/2}.$$

It follows by Cauchy-Schwarz and \eqref{useplu} that after rearranging \eqref{3eq} so that there are only plus sings on both sides,  one has
$$
K^2|A| \gg |A|^3/|A+A|^{3/2},
$$
which completes the proof. Indeed,  if $|A+A| > |A|^{4/3}$, the claim holds trivially, otherwise $|A|^2|A+A|<p^2$, and Theorem \ref{main} applies if $|A|<p^{3/5}.$

Return to equations \eqref{useplu} and  \eqref{3eq} we may now assume, by analogue with the former one, that $A$ is such that
$$
|A+A^2+A^2| \ll |A+A^2|^2/|A|,
$$
and in the former one replace the terms $c^2,{c'}^2$ with just $c,c'\in A$. It follows from the same as above energy argument that
$$
|A+A^2+A^2| |A+A|^{3/2} \gg |A|^3.
$$

Combining the last two inequalities with  the Pl\"unnecke inequality in the form (\ref{2pl})
$$
|A+A|\ll |A+A^2|^2/|A|,
$$
we arrive in
$$
|A+A^2|\gg |A|^{11/10},
$$
under the same constraint $|A|<p^{3/5}.$
\end{proof}

Recall that claim (b) of Proposition \ref{prop:5} states that for $|A|<p^{3/5}$, one has \[ |A-A|^3 |A^3+A^3| \gg |A|^{17/4},\qquad |A+A^3|\gg |A|^{29/28}.\]

  \begin{proof}[Proof of Claim (b)]
To deal with cubes, we use the discrete derivative and complete a square. The former aspect necessitates $A-A$ appearing in the estimate; to replace it with $A+A$ one would have to use Pl\"unnecke-Ruzsa inequalities.

Suppose, $|A|>C$ for some large enough absolute $C$. First, refine $A$ to a large subset $A''$, as follows.

Choose a large $A'\subseteq A$, such that
\begin{equation}
|{A'}^3+{A'}^3+{A'}^3|\ll |{A}^3+{A}^3|^2/|A|,
\label{threesome}
\end{equation}
by the Pl\"unnecke-Ruzsa inequality \eqref{f:plunnecke-ruzsa1'}.

Then choose a large subset $A''\subseteq A'$, such that, after another application of the Pl\"unnecke-Ruzsa inequality \eqref{f:plunnecke-ruzsa1'}, one has
\begin{equation}
|-A''+A'+A'+A'|\ll |A-A|^3/|A|^2.
\label{foursome}
\end{equation}

Reset the notation $A''$ back to $A$.

Let $b=a+d$, where $a,b\in A$, $d\in A-A$. Then
\begin{align}
\label{diff}
b^3-a^3 &= (a+d)^3 - a^3 = 3d(a^2 + ad + (1/3)d^2)\\
&= 3d[ (a+d/2)^2 + (1/12) d^2 ] = 3d[ s^2 + (1/12)d^2],     \nonumber
\end{align}
where $s = a + d/2.$

Let  $\E$ be the number of solutions of the following equation:
\begin{equation}\label{refr}
c^3+ b^3-a^3  =  {c'}^3 + {b'}^3 - {a'}^3,\qquad a,\ldots,c'\in A,\end{equation}

Then, using \eqref{threesome} and recalling that $A$ is now $A''\subseteq A'$, the Cauchy-Schwarz inequality yields
$$
|{A}^3+ {A}^3|^2/|A| \gg |A|^6/\E,
$$
once the terms in \eqref{refr} have been rearranged, so that there are plus signs only.

By \eqref{diff} the left-hand side of the equation defining $\E$ can be rewritten as
$$
(c^3+ d^3/4) + 3d s^2 =\ell (s^2),
$$
where the lines  $\ell$ in a set $L_P$ have equations
$$
y = 3d \,x + (d^3/4+c^3): \; d\in A-A, \,c\in A.
$$
The lines in $L_P$ are defined by covectors $(3d, d^3/4+c^3)$, which lie on the translates of the same cubic through the origin, thus the number of collinear ones is $O(|A-A|)$.

The cardinality of the set of values of $s=a+d/2$,
is at most $|A+A+A-A|,$ whose size (in view oh the refinements done) is bounded by the right-hand side of \eqref{foursome}.

By Theorem \ref{mish} (with $O(1)$ weights hidden in the $\ll$ symbols) we have, given that
$|A-A|^4/|A|<p^2$,
$$
\E\ll|A-A|^{6}  |A|^{-3/2} + |A-A|^7|A|^{-3}.
$$
If the second term dominates, there is nothing to prove, so we conclude that for $|A|<p^{3/5}$
\begin{equation}
|{A}^3+ {A}^3| |A-A|^{3}  \gg |A|^{17/4}.
\label{ceq}\end{equation}
Indeed, if $|A-A|\geq |A|^{13/12}$ the latter inequality holds trivially; otherwise the condition $|A-A|^4/|A|<p^2$ is satisfied.

Returning to  equation \eqref{refr},  we set $c,c'\in A$ and repeat the argument, resulting in
$$
|{A}+{A}^3| |A-A|^{3}  \gg |A|^{17/4},
$$
By the Ruzsa distance inequality,
$$
|A-A|\ll |A+A^3|^2/|A|,
$$
so we  have  $|A+A^3|\gg |A|^{29/28}$.
\end{proof}

\subsubsection{Proof of Corollary \ref{cor:4}}

Recall that Corollary \ref{cor:4} states that for $|A|<p^{8/15}$ and $|A|<p^{7/12},$ respectively, one has
\[
|(A-A)^2 +(A-A)^2| \gg |A|^{9/8},\qquad |(A-A)^3 +(A-A)^3|\gg |A|^{36/35}.
\]

\begin{proof}
     To prove the first inequality, observe that $|\Delta:=(A-A)^2 + (A-A)^2| \geq |A-A|$ trivially. By the Pl\"unnecke inequality \eqref{fatpl}
   \begin{equation}\label{au}
   |A-A+A-A|\leq K^4|A|,
\end{equation}
  where $|A-A|=K|A|.$

Applying the first inequality claimed by Proposition \ref{prop:5}  to the set $A-A$ we have, by \eqref{au}:
$$
   |\Delta|^4 K^{12} |A|^3 \gg  K^8|A|^8,$$
   hence $|\Delta|\gg |A|^{5/4}/K$. Optimising with $|\Delta|\geq K|A|$
   leads to the claim $|\Delta|\gg|A|^{9/8}.$  If one looks back to the proof of Proposition \ref{prop:5}, the constraint in terms of $p$ there becomes $|A-A|^2|A-A+A-A|<p^2$, which holds if
  $K^6|A|^3 < p^2.$ If $K>|A|^{1/8}$ the claim of this corollary holds trivially. Otherwise we arrive in the condition $|A|<p^{8/15}$  as stated.

\medskip
To prove the second inequality,  set $D=A-A$. From \eqref{ceq} and using \eqref{au} we have
  $$
  |D^3+D^3|\gg |D|^{17/4}/|D-D|^3 \gg K^{-31/4}|A|^{5/4}.
  $$
  Optimising it with the other trivial lower bound $K|A|$ we have  $|D^3+D^3|\gg |A|^{36/35}.$ The estimate is valid for $|A|<p^{7/12}$.
   Indeed, estimate \eqref{ceq} with $A-A$ replacing $A$ is valid for $|A-A|<p^{3/5}$. If $|A-A|\geq |A|^{36/35}$ our claim about cubic distances holds trivially. Otherwise $|A|<p^{7/12}$ implies
   $|A-A|<p^{3/5}$.
\end{proof}

\subsection{Sums of reciprocals}
\label{sec:sums-reciprocals-1}
\subsubsection{Proof of Proposition \ref{prop:8}}
Recall that Proposition \ref{prop:8} states that for  $A\in F_*$, with  $|A|<p^{5/8}$ one has
$$
|A+A|+|1/A+ 1/A| \gg|A|^{16/15}, \qquad |A+A^{-1}| \gg |A|^{31/30}.
$$
\begin{proof}  Following the proof of Theorem 4.1 in \cite{Bo}, we denote $X=(A^{-1}+A^{-1})^{-1}$, $S=(A+A)^{-1}$, so for $a,b\in A,\,a-\frac{a^2}{a+b}\in X$. Applying Cauchy-Schwarz yields
\begin{equation}
|\{(a,a',s)\in A\times A\times S:\,(a-sa^2,a'-s{a'}^2)\in X\times X\}|\geq|A|^4/|S|.
\label{qt}\end{equation}
The left-hand side can be interpreted as the number of incidences between $X\times X$ and the set of $O(|A|^2)$ lines with equations
$$y= \frac{{a'}^2}{a^2} x + a'\left(1-\frac{a'}{a}\right).$$
Assuming that $|X|<p^{2/3}$ we apply the incidence bound of Corollary \ref{cor:n}, which together with \eqref{qt} yields
$$
|A|^{8/3} \ll |X|^{3/2} |A+A| \,.
$$
Note that if $|A|<p^{5/8}$, the constraint $|X|<p^{2/3}$  may  be violated only if $|X|>|A|^{16/15}$; otherwise the latter estimate holds and implies the first inequality claimed by the proposition.

The second inequality follows by the Pl\"unnecke-Ruzsa inequality \eqref{2pl}, since both
$$
|A+A|,\; |1/A+ 1/A|\;\leq\; |A+A^{-1}|^2/|A|.
$$\end{proof}

\subsubsection{Proof of Corollary \ref{cor:x}}

Recall that Corollary \ref{cor:x} states that for $A\subset F_*$, with  $|A|<p^{5/8}$, and any $\alpha \neq 0$ one has
$$
    |A \cap \alpha/A| \ll |A+A|^{15/16}.
$$

\begin{proof}
We set
$A' = A \cap \alpha/A$.
Then $A' + A' \subseteq A+A$ and $\alpha/A' + \alpha/A' \subseteq A+A.$ The claim follows by applying 
Proposition \ref{prop:8} to $A'$.
\end{proof}

\subsubsection{Proof of Proposition \ref{prop:6}}

Recall that Proposition \ref{prop:6} states that for  $A\subset F_*$, with $|A|<p^{8/15}$, one has \[ |1/A + 1/(A+A)| \gg |A|^{9/8}.\]

\begin{proof}
We let $X=(A^{-1}+A^{-1})^{-1}$, $S=(A+A)^{-1}$ as in the previous proof, as well as $Y:=A^{-1}+S$.
Since in general there is no inclusion of $A^{-1}$ into $S$, we note the bound
\begin{equation}\label{aux}
  |X| \leq \frac{|A^{-1}+S|^2}{|S|}=\frac{|Y|^2}{|S|},
\end{equation}
by the Pl\"unnecke inequality \eqref{2pl}.

Let $\E$ be the number of solutions of the following equation, with variables in $A$:
$$
\frac{1}{c}+ \frac{1}{a+b}=\frac{1}{c'} + \frac{1}{a'+b'},
$$
where we assume that $a+b,a'+b'\neq 0.$

Note that with
$$x=(a^{-1}+b^{-1})^{-1} =  b - \frac{b^2}{a+b},$$
one has $\frac{1}{a+b}  = b^{-1}- b^{-2}x.$

Hence $\E=\E(L,X)$, where $L=L_P,$
with
$$
P=\{ (-b^{-2}, b^{-1}+c^{-1}):\, b.c\in A\}.
$$
To estimate the quantity $\E(L,X)$ we apply Theorem \ref{mish}, the  point set in question being $Q=P\times X$, see also the proof of Proposition \ref{prop:5}. Similarly to the latter case, here a point/plane can also have an integer weight, but only up to $2$, and this gets absorbed into constants.

Assuming $|A|^2|X|<p^2$ we once again arrive in the bound \eqref{app} for $\E$, with $X$ replacing $A+A$ in the right-hand side. The first term in the estimate  dominates, or there is nothing to prove. Indeed, the converse case implies that $|X|\gg |A|^2$, in which case $Y$ is also large, by \eqref{aux}.

Noting once more \eqref{aux} we now have
$$\E\ll |A|^{3} |Y|^{3}/|A+A|^{3/2}.$$

Hence, by Cauchy-Schwarz,
$$
    |Y|\gg |A|^6/\E
$$
or, equivalently,
$$
 |Y|^4\gg |A|^3|A+A|^{3/2}\geq |A|^{9/2},
$$
with the trivial estimate for $|A+A|$ or $|S|$.

We conclude that
$|Y|\gg |A|^6/\E \gg |A|^{9/8}$, given that $|A|< p^{8/13}$. Indeed, the latter condition in view of \eqref{aux} ensures the constraint $|A|^2|X|<p^2$ for the above application of the incidence estimate of Theorem \ref{mish}.
\end{proof}

\section{Acknowledgements}

The fourth  author was supported by grant Russian Scientific  Foundation RSF 14-11-00433.

\end{document}